\documentclass{tac}

\usepackage{amsmath}

\title{Algebraic coherators, controlled theories, and Grothendieck realizations}

\address{Unaffiliated\\
 Chicago, Illinois, USA}

\author{Johnathon Taylor}

\usepackage{tikz}
\usepackage{tikz-cd}

\newtheoremrm{construction}{Construction}
\newtheoremrm{conjecture}{Conjecture}

\begin{document}

\maketitle

\begin{abstract}
We introduce a construction of algebraic coherators for Grothendieck $\infty$-groupoids using the algebraic small object argument, replacing previous approaches we have used based on distributive series of monads with a more direct method for freely adjoining coherence data. Given a controlled theory, we define unreduced and reduced Grothendieck realizations, producing $\infty$-Lawvere theories and extending this construction functorially to connected diagrams of controlled theories. We apply this framework to construct globular models for monoidal $\infty$-groupoids, symmetric monoidal $\infty$-groupoids, coherent $\infty$-groups, and Picard $\infty$-groupoids. We define canonical semi-model structures on categories of models over $\infty$-Lawvere theories and formulate a generalized pushout conjecture that implies the existence of these semi-model structures and the Homotopy Hypothesis for Grothendieck $\infty$-groupoids.
\end{abstract}

\copyrightyear{2026}



\eaddress{jt3theend17@gmail.com}

\tableofcontents

\section{Introduction}
In this paper, we revisit and refine several constructions from our previous work. In earlier papers and in our thesis \cite{Taylor2026}, we developed a program for studying Grothendieck $\infty$-groupoids. Our previous work focused on constructing theories of Grothendieck $\infty$-groupoids and higher variants of algebraic theories whose underlying objects are Grothendieck $\infty$-groupoids. Following helpful referee reports from the editors and reviewers at \emph{Applied Categorical Structures}, we came to the conclusion that Cheng's theory of distributive series of monads \cite{Cheng2011} is not the appropriate technical framework for these constructions for two fundamental reasons. First, it is not clear that the required constructions can be carried out. Second, even if they can, the resulting indexing becomes prohibitively complicated and obscures the underlying ideas. In this paper, we replace this approach with one based on Garner's theory of algebraic weak factorization systems \cite{Garner2009}. This alternative provides a cleaner, foundation for the constructions originally presented in our thesis.

In Section~2, we review the theory of Grothendieck $\infty$-groupoids. Grothendieck originally proposed the existence of these objects in \emph{Pursuing Stacks} \cite{Gr1}. Their modern development was initiated by Georges Maltsiniotis and has since been expanded by Ara, Bourke, Henry, Lanari, and Maltsiniotis, among others \cite{Dim1,Henry2016,HenLan,LanE,Geo2,Bo2}. We conclude the section by constructing the \emph{algebraic coherator}, which coincides with Ara's reduced coherator from Example~2.12 of \cite{Dim1}.

In Section~3, we combine the data of globular theories with an arbitrary limit sketch. We then introduce a notion of generalized contractibility that is suited to the construction of the Grothendieck realization developed in the following section. We conclude by proving that freely adjoining lifts is a well-behaved operation satisfying a universal property that plays a central role in Section~4.

In Section~4, we combine the initial finite product theory of \cite{Law2} with the algebraic coherator constructed in Section~4 to define the category of $\infty$-Lawvere theories. For every controlled theory $\Omega$, we construct an algebraic weak factorization system on an appropriate subcategory of $\infty$-Lawvere theories. We define the \emph{unreduced Grothendieck realization} of a controlled theory to be the $\infty$-Lawvere theory obtained by applying the fibrant replacement monad associated with the induced algebraic weak factorization system to the initial object. We further show that this construction is functorial in the controlled theory. We then define the \emph{reduced Grothendieck realization} by identifying the duplicate choice of lift introduced in the $\infty$-groupoid component of the Unreduced Grothendieck realization. Finally, we prove that both constructions extend naturally to functors defined on connected diagrams of controlled theories.

In Section~5, we describe the homotopy theory associated with $\infty$-groupoids and models of $\infty$-Lawvere theories. Specifically, we identify classes of weak equivalences, generating cofibrations, and generating trivial cofibrations. We conclude by describing the conjectural semi-model structures on the categories of $\infty$-groupoids and models of $\infty$-Lawvere theories. The existence of the semi-model structure for $\infty$-groupoids is known to follow from work of Henry and Lanari under Henry's pushout conjecture \cite{Henry2016,HenLan,LanE}. Moreover, Henry's conjecture implies the Homotopy Hypothesis \cite{Henry2016}. We formulate an analogous conjecture for $\infty$-Lawvere theories generated as a Grothendieck realization, which we call the \emph{Generalized Pushout Conjecture}. We obtain the following theorem.

\emph{Theorem: Assuming the Generalized Pushout Conjecture, the semi-model structure exists on $\infty\mathbf{Gpd}$ and the Homotopy Hypothesis holds. Moreover, the semi-model structure exists on $\mathbf{Mod}(\mathbf{UGR}(J))$ for every connected diagram of controlled theories.}

\subsection*{Background Assumptions and Notation}
We assume the reader is familiar with the material from the controlled theories section of my previous paper \cite{Taylor2026Controlled}, the algebraic small object argument of Garner \cite{Garner2009}, and the theory of limit sketches and the theories and extensions out of them. We list some notation here as not to have to create a separate section of the paper.

\begin{notation}
We write the following notation.
\begin{itemize}
    \item Given a limit sketch $E$, we write $\mathbf{Th}_E$ and $\mathbf{Ex}_E$ for the categories of theories and extension over $E$, respectively.
    \item We write $\mathbf{Top}$ for a convenient category of spaces.
    \item We write $\mathbf{Cat}$ for the category of small categories.
\end{itemize}
\end{notation}

\subsection*{Similar Work} 
Our work is similar to the work of Bressie \cite{Bressie2020}. In \cite{Bressie2020}, Bressie defines a methodology to categorify the syntactic structure for pros and uses the globular operads of Leinster \cite{Lein} as a backdrop. Our construction here is seemingly more general than the construction of Bressie and we use Grothendieck $\infty$-groupods as a backdrop.

\subsection*{Future Work}
As was stated in the future work section of \cite{Taylor2026Controlled}, we are moving towards building a bridge between categorified models of controlled theories. In particular, we will write a short paper on how to construct weak orientals in the category of Grothendieck $\infty$-groupoids.
 We study the induced \emph{nerve construction}. We ask whether the hypothetical canonical semi-model structure, if it exists, is transferred off of simplicial sets.  Following that, we will write the paper promised in \cite{Taylor2026Controlled} to build Quillen equivalences of semi-model categories between categorified models of controlled theories and the model categories of algebras over the nerve realization of \cite{Taylor2026Controlled} generated by those same controlled theories. 

\section{Grothendieck infinity groupoids}
\subsection*{Background}

This subsection reviews the basic theory of globular theories following
Maltsiniotis~\cite{Malt}, Bourke~\cite{Bo2}, and Ara~\cite{Dim1}. We include the necessary material to establish notation.

Globular sets provide a combinatorial framework for describing higher-dimensional cells together with their source and target operations. Their indexing category is the globe category.

\begin{notation}
The \emph{globe category} $\mathbf{G}$ is the category
\[
\cdots
\]
subject to the relations
\[
ss=ts,\qquad tt=st.
\]
\end{notation}

\begin{definition}
A \emph{globular object} in a category $C$ is a functor
\[
X:\mathbf{G}^{op}\longrightarrow C.
\]
\end{definition}

\begin{example}\label{example_important}
For each $n\ge0$, let
\[
D^n=\{x\in\mathbf R^n:\|x\|\le1\}.
\]
The assignments
\[
\sigma^n(x)=\left(x,\sqrt{1-\|x\|^2}\right),
\qquad
\tau^n(x)=\left(x,-\sqrt{1-\|x\|^2}\right)
\]
define a functor
\[
D:\mathbf G\to\mathbf{Top},
\]
where $\sigma^n$ and $\tau^n$ include the upper and lower hemispheres of
$D^{n-1}$ into $D^n$. Consequently every space $X$ determines a globular object
\[
\mathbf{Top}(D^{(-)},X).
\]
\end{example}

To describe higher-dimensional operations we require certain canonical limits of globular objects.

\begin{definition}
A \emph{table of dimensions} is an odd-length sequence
\[
\vec n=(n_1,\ldots,n_{2k+1})
\]
whose entries determine a zig-zag diagram in $\mathbf G$
\end{definition}

\begin{definition}
The \emph{height} of a table of dimensions $\vec n$ is
\[
\mathbf{hgt}(\vec n)=\max_i n_i.
\]
\end{definition}

\begin{definition}
Let $A:\mathbf G^{op}\to C$ be a globular object and let $\vec n$ be a table of dimensions.
The limit of the induced zig-zag diagram in $C$, when it exists, is called the
\emph{globular product} of $(A,\vec n)$.
\end{definition}

Globular products play the role of finite products for globular theories; they encode the pasting shapes on which higher-dimensional operations are defined.

The category $\Theta_0$ is obtained by freely adjoining globular products to the globe category.

\begin{definition}
The category $\Theta_0$ has tables of dimensions as objects and
\[
\Theta_0(\vec n,\vec m)
=
[\mathbf G^{op},\mathbf{Set}]
(Y(\vec n),Y(\vec m)),
\]
where $Y$ denotes the Yoneda embedding.
\end{definition}

Every table of dimensions determines a globular product in
$\Theta_0^{op}$, so $\Theta_0^{op}$ admits globular products.

\begin{lemma}
There is a functor
\[
D:\mathbf G^{op}\longrightarrow\Theta_0^{op}
\]
that is universal among globular-product-preserving extensions. Explicitly, if
$A:\mathbf G^{op}\to C$ admits globular products, then there exists an essentially unique globular-product-preserving extension
\[
A':\Theta_0^{op}\to C.
\]
Moreover, $A'(\vec n)$ is the globular product associated to $\vec n$.
\end{lemma}

\begin{proof}
This is Lemma 2.1 of \cite{Bo2} from Bourke.
\end{proof}

We now introduce the lifting property that characterizes contractible globular theories.

\begin{definition}
Let $A$ be a globular object. An \emph{admissible pair} of dimension $n$ consists of parallel maps
\[
\begin{tikzpicture}[node distance=2cm]
    \node (A) {$X$};
    \node (B) [right of=A] {$A(n)$};
    \draw[transform canvas={yshift=0.3ex},->]
        (A) to node[above=3] {$f$} (B);
    \draw[transform canvas={yshift=-0.3ex},->,swap]
        (A) to node[below=3] {$g$} (B);
\end{tikzpicture}
\]
where $X$ is a globular product of height at most $n+1$ and either $n=0$ or the maps have common source and common target.
\end{definition}

\begin{definition}
A \emph{lift} of an admissible pair
$(f,g)$
is a morphism
\[
\delta_{f,g}:X\to A(n+1)
\]
whose source and target are $f$ and $g$, respectively.
\end{definition}

\begin{definition}
A globular object is \emph{contractible} if every admissible pair admits a lift.
A globular theory is contractible when regarded as a globular object.
\end{definition}

Grothendieck's approach constructs contractible globular theories inductively by freely adjoining lifts to chosen admissible pairs.

\begin{definition}
An $(\infty,0)$\emph{-coherator} is a contractible globular theory obtained as the colimit of a sequence
\[
\Theta_0^{op}=C_0
\longrightarrow
C_1
\longrightarrow
C_2
\longrightarrow\cdots,
\]
where each morphism freely adjoins lifts for a chosen collection of admissible pairs.
\end{definition}

\begin{notation}
If $C$ is an $(\infty,0)$-coherator, we write
\[
\infty\mathbf{Gpd}_C
=
\mathbf{Mod}_{\Theta_0^{op}}(C)
\]
for the category of globular-product-preserving functors
$C\to\mathbf{Set}$.
Its objects are called \emph{Grothendieck $\infty$-groupoids}.
\end{notation}

We add the following notation for later use to end this subsection.

\begin{notation}\label{notation_for_induced_ids}
Let $C$ be an $(\infty,0)$-coherator. For all $k\geq 0$, define $Z^k:0\to k$ to be the iterated composite of the lift $Z:i\to i+1$ for the admissible pair consisting of the identity and itself. By the universal property of the globular products, there is a uniquely induced map $Z^{\vec{p}}:0\to\vec{p}$ for all $\vec{p}\in\textbf{ob}(\Theta_0^\mathbf{op})$. 
\end{notation}

\subsection*{Factorization system for globular theories}

In this subsection we construct the algebraic weak factorization system that will generate a theory we use for the rest of the paper. The generating maps encode the operation of freely adjoining lifts to admissible pairs. Applying Garner's algebraic small object argument then produces a fibrant replacement monad whose fibrant replacement of the initial globular theory is an $(\infty,0)$-coherator. 

\begin{definition}\label{spheres_in_inf_gpds}
Let $\vec{p}$ be a table of dimensions and let $k\geq0$ satisfy
\[
\mathbf{hgt}(\vec{p})\leq k+1.
\]
The \emph{$(\vec{p},k)$-sphere} is the globular theory
\[
S_{\vec{p},k}
\]
obtained from $\Theta_0^\mathbf{op}$ by freely adjoining an admissible pair
\[
\begin{tikzpicture}[node distance=2cm]
    \node (A) {$\vec{p}$};
    \node (B)[right of=A]{$k$};
    \draw[transform canvas={yshift=0.3ex},->]
        (A) to node[above=3] {$f$} (B);
    \draw[transform canvas={yshift=-0.3ex},->,swap]
        (A) to node[below=3] {$g$} (B);
\end{tikzpicture}
\]
subject to the relations
\[
s\circ f=s\circ g
\qquad\text{and}\qquad
t\circ f=t\circ g
\]
whenever $k\geq1$. Equivalently, $S_{\vec{p},k}$ is the free globular theory containing a distinguished admissible pair of dimension $k$ with domain $\vec{p}$. There is a canonical inclusion
\[
\Theta_0^\mathbf{op}\hookrightarrow S_{\vec{p},k}.
\]
\end{definition}

\begin{definition}\label{disks_in_inf_gpds}
The \emph{$(\vec{p},k)$-disk}
\[
D_{\vec{p},k}
\]
is obtained from $S_{\vec{p},k}$ by freely adjoining a lift
\[
\delta_{f,g}:\vec{p}\longrightarrow k+1
\]
of the distinguished admissible pair so that the diagram
\[
\begin{tikzpicture}[node distance=2cm]
    \node (A) {$\vec{p}$};
    \node (B)[right of=A]{$k$};
    \node (C) [above of=B]{$k+1$};
    \draw[transform canvas={yshift=0.3ex},->]
        (A) to node[above=3] {$f$} (B);
    \draw[transform canvas={yshift=-0.3ex},->,swap]
        (A) to node[below=3] {$g$} (B);
    \draw[transform canvas={xshift=-0.3ex},->]
        (C) to node[left=3] {$s$} (B);
    \draw[transform canvas={xshift=0.3ex},->,swap]
        (C) to node[right=3] {$t$} (B);
    \draw[->]
        (A) to node[above=4] {$\delta_{f,g}$} (C);
\end{tikzpicture}
\]
commutes. Thus $D_{\vec{p},k}$ is the free globular theory obtained by specifying a filler for the distinguished admissible pair. It is equipped with the canonical inclusion
\[
\Theta_0^\mathbf{op}\hookrightarrow D_{\vec{p},k}.
\]
\end{definition}

\begin{notation}\label{incl_of_sphere_into_disks}
For every table of dimensions $\vec{p}$ and every $k\geq0$, there is a canonical inclusion
\[
j_{\vec{p},k}:S_{\vec{p},k}\longrightarrow D_{\vec{p},k}.
\]
We write
\[
I=
\left\{
j_{\vec{p},k}
:
S_{\vec{p},k}\rightarrow D_{\vec{p},k}
\;\middle|\;
\vec{p}\in\mathbf{ob}(\Theta_0^\mathbf{op}),\;
k\ge0
\right\}.
\]
\end{notation}

\begin{lemma}\label{admiss_1}
Each object $S_{\vec{p},k}$ and $D_{\vec{p},k}$ is presentable in
$\mathbf{Th}_{\Theta_0^\mathbf{op}}$.
Moreover, the set $I$ is admissible for Garner's algebraic small object argument.
\end{lemma}

\begin{proof}
This follows from Subsection~3.11 and Lemma~3.12 of \cite{Malt}.
\end{proof}

\begin{notation}
Let
\[
(L_I,E_I,R_I,\delta_I)
\]
denote the algebraic weak factorization system on
$\mathbf{Th}_{\Theta_0^\mathbf{op}}$
cofibrantly generated by the set $I$.
\end{notation}

The fibrant replacement of the initial globular theory with respect to this algebraic weak factorization system recovers the algebraic coherator.

\begin{lemma}
Let
\[
J^{AC}:\Theta_0^\mathbf{op}\longrightarrow AC
\]
be the fibrant replacement of the initial object
\[
\mathrm{id}_{\Theta_0^\mathbf{op}}
\]
with respect to the algebraic weak factorization system
$(L_I,E_I,R_I,\delta_I)$.
Then $AC$ is an $(\infty,0)$-coherator.
\end{lemma}

\begin{proof}
The proof is obtained by adapting the argument of Theorem 3.14 of \cite{Malt}.
\end{proof}

We refer to $AC$ as the \emph{algebraic coherator}. It coincides with the reduced coherator of Ara Example 2.12 \cite{Dim1}.

\section{Globularizing limit sketches, theories, and extensions} \label{glob_lim_sketch} 
\subsection*{The Setup}

Throughout this section, let $E$ be a realized limit sketch. Our goal is to combine globular theories with $E$-theories in a manner that preserves the structures arising from each independently. We first construct an intermediate category using the funny tensor product and then freely complete it under the required globular products and distinguished limits. 

\begin{definition}\label{glob_of_lim_sk}
The \emph{globularization} of $E$ is the realized limit sketch
\[
\Theta_0^\mathbf{op}\bullet E
\]
obtained as the universal realization of the product limit sketch
\[
\Theta_0^\mathbf{op}\times E,
\]
whose specified cones are the Cartesian product of the specified cones of
$\Theta_0^\mathbf{op}$ and $E$.
\end{definition}

The terminology \emph{globularization} follows \cite{Bressie2020}. Throughout this section we write
$\odot$ for the funny tensor product of categories.

Our next construction combines a globular theory with an $E$-theory by identifying the two structures only where required. Away from these specified identifications, the construction behaves like the funny tensor product.

\begin{construction}\label{glob_of_lim_sk_th}
Let
\[
H:E\to L
\]
be an $E$-theory and let
\[
K:\Theta_0^\mathbf{op}\to B
\]
be a globular theory. Define
\[
B\diamond L
\]
to be the colimit of the following diagram, where $\kappa_1$, $\kappa_2$, and $\kappa_3$ denote the canonical colimit maps.
\[
\begin{tikzcd}
\Theta_0^\mathbf{op}\odot E
\ar[dd,swap,"K\odot1_E"]
\ar[dr,swap,"q"]
\ar[rr,"1_{\Theta_0^\mathbf{op}}\odot H"]
&&
\Theta_0^\mathbf{op}\odot L
\ar[rd,"q"]
\ar[dd,"K\odot1_L" near start]
&
\\
&
\Theta_0^\mathbf{op}\times E
\ar[rr,"1_{\Theta_0^\mathbf{op}}\times H" near start]
\ar[dd,"K\times1_E" near start]
&&
\Theta_0^\mathbf{op}\times L
\ar[dd,"\kappa_1"]
\\
B\odot E
\ar[rr,"1_B\odot H" near start]
\ar[dr,swap,"q"]
&&
B\odot L
\ar[dr,"\kappa_3"]
&
\\
&
B\times E
\ar[rr,swap,"\kappa_2"]
&&
B\diamond L.
\end{tikzcd}
\]
\end{construction}

The above construction is functorial.

\begin{lemma}
Construction~\ref{glob_of_lim_sk_th} extends to a functor
\[
\diamond:
\mathbf{Th}_{\Theta_0^\mathbf{op}}
\times
\mathbf{Th}_E
\longrightarrow
(\Theta_0^\mathbf{op}\times E)/\mathbf{Cat}.
\]
\end{lemma}

\begin{proof}
This follows immediately from the universal property of the colimit.
\end{proof}

The category $B\diamond L$ generally does not possess the globular products and distinguished limits required for our applications. We therefore freely complete it under these structures.

\begin{notation}\label{Uni_lim_compl_of_glob}
Given an $E$-theory
\[
H:E\to L
\]
and a globular theory
\[
K:\Theta_0^\mathbf{op}\to B,
\]
let
\[
\lambda_{B,L}:B\diamond L\longrightarrow B\bullet L
\]
denote the universal completion of $B\diamond L$ under globular products in the first variable and the specified limits of $E$ in the second.
\end{notation}

The defining universal property is the following.

\begin{lemma}
There is a category $B\bullet L$ together with a functor
\[
\lambda_{B,L}:B\diamond L\longrightarrow B\bullet L
\]
that preserves globular products in the first variable and the specified limit cones of $E$ in the second. Moreover, if
\[
R:B\diamond L\longrightarrow V
\]
preserves these structures, then there exists a unique functor
\[
R':B\bullet L\longrightarrow V
\]
preserving globular products in the first variable and the specified limit cones of $E$ in the second such that
\[
R'\circ\lambda_{B,L}=R.
\]
\end{lemma}

It is important to distinguish the constructions
\[
B\diamond L
\qquad\text{and}\qquad
B\bullet L.
\]
The category $B\diamond L$ is assembled from copies of
\[
B\times E
\qquad\text{and}\qquad
\Theta_0^\mathbf{op}\times L
\]
using the funny tensor product elsewhere, whereas $B\bullet L$ is obtained from $B\diamond L$ by freely adjoining the globular products and distinguished limits required of a globularized theory. This asymmetry is essential: it allows us to distinguish the order in which globular operations and the algebraic operations arising from a controlled theory are applied. This distinction will play a central role in the categorification that follows.

\begin{lemma}
The assignment
\[
(B,L)\longmapsto B\bullet L
\]
extends to a functor
\[
\bullet:
\mathbf{Th}_{\Theta_0^\mathbf{op}}
\times
\mathbf{Th}_E
\longrightarrow
\mathbf{Th}_{\Theta_0^\mathbf{op}\bullet E}.
\]
\end{lemma}

\begin{proof}
Suppose
\[
H:E\to L
\]
is an $E$-theory and
\[
K:\Theta_0^\mathbf{op}\to B
\]
is a globular theory. By the universal property of
\[
\lambda_{\Theta_0^\mathbf{op},E}:
\Theta_0^\mathbf{op}\diamond E
\longrightarrow
\Theta_0^\mathbf{op}\bullet E
\]
and the functoriality of $\diamond$, there exists a unique morphism
\[
K\bullet H:
\Theta_0^\mathbf{op}\bullet E
\longrightarrow
B\bullet L
\]
such that the diagram
\[
\begin{tikzcd}
\Theta_0^\mathbf{op}\diamond E
\ar[r,"\lambda_{\Theta_0^\mathbf{op},E}"]
\ar[d,swap,"K\diamond H"]
&
\Theta_0^\mathbf{op}\bullet E
\ar[d,"K\bullet H"]
\\
B\diamond L
\ar[r,swap,"\lambda_{B,L}"]
&
B\bullet L
\end{tikzcd}
\]
commutes.

By abuse of notation, we continue to denote by
\[
K\bullet H
\]
the composite
\[
(K\bullet H)\circ
\lambda_{\Theta_0^\mathbf{op},E}.
\]
Functoriality follows immediately from the universal property of $\bullet$ together with the functoriality of $\diamond$.
\end{proof}

\begin{definition}
Given an $E$-theory
\[
H:E\to L
\]
and a globular theory
\[
K:\Theta_0^\mathbf{op}\to B,
\]
composition with
\[
K\bullet H
\]
defines the pullback functor
\[
(K,H)^*:
\mathbf{Th}_{B\bullet L}
\longrightarrow
\mathbf{Th}_{\Theta_0^\mathbf{op}\bullet E},
\]
given on objects by
\[
(K,H)^*(R)=R\circ(K\bullet H).
\]
\end{definition}

We now isolate the portion of a $(B\bullet L)$-theory determined by the underlying $E$-structure.

\begin{definition}\label{the_zero_cell_level}
Let
\[
R:B\bullet L\to V
\]
be a $(B\bullet L)$-theory. The \emph{$0$-cell level} of $V$ is the full subcategory determined by the image of the functor
\[
(0,-):L\longrightarrow V.
\]
\end{definition}

\begin{lemma}
The map
\[
(0,-):E\longrightarrow B
\]
endows the $0$-cell level with the structure of an $E$-theory. Moreover,
\[
(0,-):L\longrightarrow V
\]
is a morphism of $E$-theories onto the $0$-cell level.
\end{lemma}

\begin{notation}\label{full_sub_iso_on_full_image}
We write
\[
\mathbf{Th}^{\cong}_{B\bullet L}
\]
for the full subcategory of
\[
\mathbf{Th}_{B\bullet L}
\]
whose objects are those $(B\bullet L)$-theories
\[
R:B\bullet L\to V
\]
for which
\[
(0,-):L\longrightarrow V
\]
is an isomorphism of $E$-theories onto the $0$-cell level.
\end{notation}

The preceding construction admits an analogous version over an arbitrary object of $L$.

\begin{definition}
Let
\[
R:B\bullet L\to V
\]
be a $(B\bullet L)$-theory and let $w\in\mathbf{Ob}(L)$. The \emph{$w$-slice} of $V$ is the full subcategory determined by the image of the functor
\[
(-,w):B\longrightarrow V.
\]
\end{definition}
\subsection*{Generalized Contractibility}
We begin by introducing a generalized notion of contractibility. The first step is to associate to every morphism two canonical morphisms whose source and target both lie in dimension~$0$.

Throughout this section, fix a limit sketch $E$, a category $S$ with
$\mathbf{Ob}(S)=\mathbf{Ob}(E)$, two $E$-theories $L$ and $L'$, a faithful identity-on-objects functor
\[
\phi:S\to L,
\]
a full morphism of $E$-theories
\[
\gamma:L\to L',
\]
and an $(\infty,0)$-coherator
\[
K:\Theta_0^\mathbf{op}\to B.
\]

\begin{definition}\label{gen_admissibility_cond}
Suppose that
\[
R:B\bullet L\to V
\]
is an object of $\mathbf{Th}_{B\bullet L}^{\cong}$ and
\[
f:(\vec{p},x)\to(n,y)
\]
is a morphism of $V$. The \emph{bottom source} and \emph{bottom target} of $f$ are the composites
\[
(0,x)\xrightarrow{(Z^{\vec{p}},x)}
(\vec{p},x)
\xrightarrow{f}
(n,y)
\xrightarrow{(s^n,y)}
(0,y)
\]
and
\[
(0,x)\xrightarrow{(Z^{\vec{p}},x)}
(\vec{p},x)
\xrightarrow{f}
(n,y)
\xrightarrow{(t^n,y)}
(0,y),
\]
respectively, where $Z^{\vec p}:0\to\vec p$ is defined in Notation~\ref{notation_for_induced_ids}. We denote these morphisms by $U_s(f)$ and $U_t(f)$, respectively.
\end{definition}

The following observation is immediate.

\begin{lemma}
Suppose that
\[
f:(\vec p,x)\to(0,y)
\]
is a morphism of $V$. Then
\[
U_s(f)=U_t(f).
\]
\end{lemma}

\begin{proof}
Since $s^0=t^0=\mathrm{id}_0$,
\[
U_s(f)
=(s^0,y)\circ f\circ(Z^{\vec p},x)
=f\circ(Z^{\vec p},x)
=(t^0,y)\circ f\circ(Z^{\vec p},x)
=U_t(f).
\]
\end{proof}

Accordingly, whenever the codomain of $f$ is $(0,y)$, we simply write
\[
U(f):=U_s(f)=U_t(f).
\]

We now define the admissible pairs that will be used throughout the remainder of this paper.

\begin{definition}
Let $R:B\bullet L\to V$ be an object of $\mathbf{Th}_{B\bullet L}^{\cong}$. An \emph{$S$-admissible pair} of $n$-cells in $V$ is a pair of morphisms
\[
\begin{tikzcd}
(\vec{p},x)
\ar[r,shift left=.75ex,"f"]
\ar[r,shift right=.75ex,swap,"g"]
&
(n,y)
\end{tikzcd}
\]
where $\mathbf{hgt}(\vec p)\le n+1$, such that either

\begin{itemize}
\item $n=0$ and
\[
U(f),U(g)\in\operatorname{im}(0,\phi(-)),
\]
or

\item $n\ge1$,
\[
(s,y)\circ f=(s,y)\circ g,
\qquad
(t,y)\circ f=(t,y)\circ g,
\]
and
\[
U_s(f),U_s(g),U_t(f),U_t(g)
\in\operatorname{im}(0,\phi(-)).
\]
\end{itemize}
\end{definition}

Thus, an $S$-admissible pair is one whose bottom source and bottom target arise faithfully from the category $S$.

\begin{definition}\label{gen_admiss_pair}
Let $R:B\bullet L\to V$ be an object of $\mathbf{Th}_{B\bullet L}^{\cong}$. An \emph{$(S,L')$-admissible pair} of $n$-cells is an $S$-admissible pair
\[
\begin{tikzcd}
(\vec{p},x)
\ar[r,shift left=.75ex,"f"]
\ar[r,shift right=.75ex,swap,"g"]
&
(n,y)
\end{tikzcd}
\]
such that either $n\ge1$ or
\[
\gamma(U(f))=\gamma(U(g)).
\]
\end{definition}

The additional condition requires the bottom morphisms to become identified after applying $\gamma$. These are precisely the admissible pairs for which we require lifts.

\begin{definition}
Let $R:B\bullet L\to V$ be an object of $\mathbf{Th}_{B\bullet L}^{\cong}$. A \emph{lift} of an $S$-admissible pair is an arrow
\[
\kappa:(\vec p,x)\to(n+1,y)
\]
such that the following diagram commutes.
\[
\begin{tikzcd}
&
(n+1,y)
\ar[d,shift left=.75ex,"(t{,}y)"]
\ar[d,shift right=.75ex,swap,"(s{,}y)"]
\\
(\vec p,x)
\ar[r,shift left=.75ex,"f"]
\ar[r,shift right=.75ex,swap,"g"]
\ar[ru,"\kappa"]
&
(n,y)
\end{tikzcd}
\]
\end{definition}

\begin{definition}\label{gen_contract_yay}
Let $R:B\bullet L\to V$ be an object of $\mathbf{Th}_{B\bullet L}^{\cong}$. We say that $R$ is \emph{$S$-contractible} if every $S$-admissible pair admits a lift. We say that $R$ is \emph{$(S,L')$-contractible} if every $(S,L')$-admissible pair admits a lift.
\end{definition}

\subsection*{Construction to add S-lifts is Well-Behaved}
```latex
\begin{definition}
Let $R:B\bullet L\to V$ be an object of $\mathbf{Th}_{B\bullet L}^{\cong}$, and let $\mathcal{U}$ be a set of $S$-admissible pairs in $V$. We write
\[
R_{\mathcal{U}}:B\bullet L\to V[\mathcal{U}]
\]
for the $(B\bullet L)$-theory obtained by freely adjoining a lift to every $S$-admissible pair in $\mathcal{U}$. We denote the induced morphism of $(B\bullet L)$-theories by
\[
\begin{tikzcd}
&B\bullet L\ar[ld,swap,"R"]\ar[dr,"R_\mathcal{U}"]&\\
V\ar[rr,swap,"e_\mathcal{U}"]&&V[\mathcal{U}].
\end{tikzcd}
\]
\end{definition}

\begin{theorem}\label{freely_add_lifts}
Let $R:B\bullet L\to V$ be an object of $\mathbf{Th}_{B\bullet L}^{\cong}$, and let $\mathcal{U}$ be a set of $S$-admissible pairs in $V$. Then $R_\mathcal{U}$ is an object of $\mathbf{Th}_{B\bullet L}^{\cong}$, and
\[
e_\mathcal{U}:V\to V[\mathcal{U}]
\]
is a morphism in $\mathbf{Th}_{B\bullet L}^{\cong}$.
\end{theorem}

\begin{proof}
Suppose
\[
\begin{tikzcd}
(\vec{p},x)\ar[r,shift left=.75ex,"f"]
\ar[r,shift right=.75ex,swap,"g"]&
(n,y)
\end{tikzcd}
\]
is an element of $\mathcal{U}$. We freely adjoin a lift
\[
\delta_{f,g}:(\vec{p},x)\to(n+1,y)
\]
satisfying
\[
(s,y)\circ\delta_{f,g}=f
\]
and
\[
(t,y)\circ\delta_{f,g}=g.
\]
Since
\[
U_s(\delta_{f,g}),\,U_t(\delta_{f,g})
\in\operatorname{im}(0,\phi(-))
\subseteq\operatorname{im}(0,-),
\]
the inclusion of the $0$-cell level is unchanged. Consequently,
\[
(0,-):L\longrightarrow V[\mathcal{U}]
\]
remains an isomorphism of $E$-theories onto the $0$-cell level of $V[\mathcal{U}]$. Hence $R_\mathcal{U}$ is an object of $\mathbf{Th}_{B\bullet L}^{\cong}$, and $e_\mathcal{U}$ is a morphism in $\mathbf{Th}_{B\bullet L}^{\cong}$.
\end{proof}

The construction above satisfies the expected universal property.

\begin{lemma}\label{uni_prop_of_ext}
Suppose
\[
\begin{tikzcd}
&B\bullet L\ar[ld,swap,"R"]\ar[dr,"K"]&\\
V\ar[rr,swap,"H"]&&W
\end{tikzcd}
\]
is a morphism in $\mathbf{Ex}_{B\bullet L}$ such that, for every $S$-admissible pair
\[
\begin{tikzcd}
(\vec{p},x)\ar[r,shift left=.75ex,"f"]
\ar[r,shift right=.75ex,swap,"g"]&
(n,y)
\end{tikzcd}
\]
in $\mathcal{U}$, there is a chosen lift
\[
\delta^W_{f,g}:(\vec{p},x)\to(n+1,y)
\]
in $W$. Then there exists a unique morphism of extensions
\[
H':V[\mathcal{U}]\to W
\]
such that
\[
H'\circ e_\mathcal{U}=H
\]
and
\[
H'(\delta_{f,g})=\delta^W_{f,g}
\]
for every $S$-admissible pair
\[
\begin{tikzcd}
(\vec{p},x)\ar[r,shift left=.75ex,"f"]
\ar[r,shift right=.75ex,swap,"g"]&
(n,y)
\end{tikzcd}
\]
in $\mathcal{U}$.
\end{lemma}

\section{Infinity Lawvere theories}\label{inf_Law_Chapter}
\subsection*{Infinity-Lawvere Theories and Infinity-groupoids Re-imagined}\label{Inf_Law_Th_in}
We begin this section by defining the category of $\infty$-Lawvere theories and showing that every model of an $\infty$-Lawvere theory has an underlying Grothendieck $\infty$-groupoid.

\begin{notation}
Let $AC$ denote the algebraic coherator and consider the identity functor
\[
1_{\mathcal{F}}:\mathcal{F}\to\mathcal{F}.
\]
We write
\[
\mathcal{F}^\infty:=AC\bullet\mathcal{F}
\]
for the universal limit sketch realization of $AC\times\mathcal{F}$.
\end{notation}

\begin{notation}\label{def_of_inf_Law_Th}
We denote the category $\mathbf{Th}_{\mathcal{F}^\infty}$ by $\infty\mathbf{Law}$ and refer to its objects as \emph{$\infty$-Lawvere theories}.
\end{notation}

\begin{example}
The identity morphism $1_{\mathcal{F}^\infty}:\mathcal{F}^\infty\to\mathcal{F}^\infty$ is an $\infty$-Lawvere theory. Moreover, $1_{\mathcal{F}}$ is the initial object of $\infty\mathbf{Law}$.
\end{example}

\begin{definition}
Given an $\infty$-Lawvere theory $R:\mathcal{F}^\infty\to V$, the \emph{category of models}, denoted
\[
\mathbf{Mod}_{\mathcal{F}^\infty}(V),
\]
is the full subcategory of
\[
[V,\mathbf{Set}]
\]
whose objects are the functors $A:V\to\mathbf{Set}$ that preserve globular products in the first variable and finite products in the second variable.
\end{definition}

The initial $\infty$-Lawvere theory recovers Grothendieck $\infty$-groupoids.

\begin{lemma}\label{initial_inf_law_theory_gives_inf_grpds}
The map
\[
(-,1):AC\to\mathcal{F}^\infty
\]
of limit sketches induces an equivalence of categories
\[
(-,1)^*:\mathbf{Mod}_{\mathcal{F}^\infty}(\mathcal{F}^\infty)\to\mathbf{Mod}_{\Theta_0^\mathbf{op}}(AC).
\]
\end{lemma}

\begin{proof}
For every model $X$, we have
\[
(-,1)^*(X)=X(-,1),
\]
which is an $\infty$-groupoid of shape $AC$ by construction. Conversely, define
\[
\gamma:\mathbf{Mod}_{\Theta_0^\mathbf{op}}(AC)\to\mathbf{Mod}_{\mathcal{F}^\infty}(\mathcal{F}^\infty)
\]
on objects by
\[
\gamma(X)(\vec{p},k)=X(\vec{p})^k.
\]
It is straightforward to verify that $(-,1)^*$ and $\gamma$ are inverse equivalences.
\end{proof}

Lemma~\ref{initial_inf_law_theory_gives_inf_grpds} shows that $\mathbf{Mod}_{\mathcal{F}^\infty}(\mathcal{F}^\infty)$ is equivalent to the category of $\infty$-groupoids of shape $AC$. This is the analogue of the classical equivalence
\[
\mathbf{Mod}_{\mathcal{F}}(\mathcal{F})\simeq\mathbf{Set}.
\]

\begin{definition}
Every $\infty$-Lawvere theory $R:\mathcal{F}^\infty\to V$ induces a functor
\[
\mathbf{Mod}_{\mathcal{F}^\infty}(R): \mathbf{Mod}_{\mathcal{F}^\infty}(V)\to \mathbf{Mod}_{\mathcal{F}^\infty}(\mathcal{F}^\infty),
\]
called the \emph{underlying $\infty$-groupoid functor}. It sends a model of an $\infty$-Lawvere theory to its underlying $\infty$-groupoid.
\end{definition}
The following result will be used in the next section.

\begin{theorem}\label{forgetful functor is monadic}
Given an $\infty$-Lawvere theory $H:\mathcal{F}^\infty\to V$, the forgetful functor
\[
U:\mathbf{Mod}(V)\to\mathbf{Mod}(\mathcal{F}^\infty)
\]
is monadic.
\end{theorem}

\subsection*{The Unreduced Grothendieck Realization}
Let
\[
\Omega=\langle \mathcal{G}:P\xrightarrow{i}\mathbf{Fr}(\mathcal{G})\xrightarrow{\mathbf{st}}L\rangle
\]
be a controlled theory, let $k\geq 0$, and let $\lambda,\delta\in P(k,1)$ satisfy
\[
\mathbf{st}(i(\lambda))=\mathbf{st}(i(\delta)).
\]
Suppose moreover that $\mathbf{hgt}(\vec{p})\leq k+1$. Define $S^{\vec{p},k}_{(\lambda,\delta)}$ to be the $(AC\bullet\mathbf{Fr}(\mathcal{G}))$-theory obtained by freely adjoining a $(P,L)$-admissible pair (Definition~\ref{gen_admiss_pair}) of the form
\[
\begin{tikzcd}
	(\vec{p},n)\ar[r,shift left=.75ex,"f"]
	\ar[r,shift right=.75ex,swap,"g"]&
	(k,1)
\end{tikzcd}
\]
such that
\[
U_s(f)=U_s(g)=(0,\lambda)
\]
and
\[
U_t(f)=U_t(g)=(0,\delta).
\]
Next, define $D^{\vec{p},k}_{(\lambda,\delta)}$ to be the theory obtained by freely adjoining a lift to this admissible pair. By Theorem~\ref{freely_add_lifts}, there is a canonical morphism
\[
j^{\vec{p},k}_{(\lambda,\delta)}:
S^{\vec{p},k}_{(\lambda,\delta)}
\longrightarrow
D^{\vec{p},k}_{(\lambda,\delta)}
\]
in $\mathbf{Th}_{AC\bullet\mathbf{Fr}(\mathcal{G})}^{\cong}$.

Let $I_\Omega$ denote the set of all morphisms
$j^{\vec{p},k}_{(\lambda,\delta)}$. Applying the algebraic small object argument to $I_\Omega$ produces a fibrant replacement monad
\[
(R_\Omega,\eta_\Omega,\mu_\Omega).
\]

\begin{definition}
Let
\[
\Omega=\langle \mathcal{G}:P\xrightarrow{i}\mathbf{Fr}(\mathcal{G})\xrightarrow{\mathbf{st}}L\rangle
\]
be a controlled theory. The \emph{unreduced Grothendieck realization} of $\Omega$ is the $\infty$-Lawvere theory
\[
\mathbf{UGR}(\Omega)
=
R_\Omega(\mathbf{id}_{AC\bullet\mathbf{Fr}(\mathcal{G})})\circ !,
\]
where
\[
!:\mathcal{F}^\infty\to AC\bullet\mathbf{Fr}(\mathcal{G})
\]
is the initial morphism in $\mathbf{Th}_{\mathcal{F}^\infty}$.
\end{definition}

Now let
\[
(I,f,H):
\left\langle
\mathcal{G}:P\xrightarrow{i}\mathbf{Fr}(\mathcal{G})
\xrightarrow{\mathbf{st}}L
\right\rangle
\longrightarrow
\left\langle
\mathcal{G}':P'\xrightarrow{i'}\mathbf{Fr}(\mathcal{G}')
\xrightarrow{\mathbf{st}'}L'
\right\rangle
\]
be a morphism of controlled theories. For brevity, write
\[
\Omega=
\left\langle
\mathcal{G}:P\xrightarrow{i}\mathbf{Fr}(\mathcal{G})
\xrightarrow{\mathbf{st}}L
\right\rangle,
\qquad
\Omega'=
\left\langle
\mathcal{G}':P'\xrightarrow{i'}\mathbf{Fr}(\mathcal{G}')
\xrightarrow{\mathbf{st}'}L'
\right\rangle,
\]
and
\[
\Phi=(I,f,H).
\]

The unreduced Grothendieck realizations of $\Omega$ and $\Omega'$ are computed as the colimits of the sequences
\[
AC\bullet\mathbf{Fr}(\mathcal{G})=\mathbf{UGR}(\Omega)_0\to \mathbf{UGR}(\Omega)_1\to\cdots
\]
and
\[
AC\bullet\mathbf{Fr}(\mathcal{G}')=\mathbf{UGR}(\Omega')_0\to \mathbf{UGR}(\Omega')_1\to\cdots,
\]
respectively.

Define
\[
\mathbf{UGR}(\Phi)_0:=1_{AC}\bullet\mathbf{Fr}(I).
\]
Suppose inductively that
\[
\mathbf{UGR}(\Phi)_n:
\mathbf{UGR}(\Omega)_n
\longrightarrow
\mathbf{UGR}(\Omega')_n
\]
has been constructed. Given a $(P,L)$-admissible pair
\[
\begin{tikzcd}
	(\vec{p},n)\ar[r,shift left=.75ex,"f"]
	\ar[r,shift right=.75ex,swap,"g"]&
	(k,1)
\end{tikzcd}
\]
in $\mathbf{UGR}(\Omega)_n$, its image
\[
\begin{tikzcd}
	(\vec{p},n)\ar[r,shift left=.75ex,"\mathbf{UGR}(\Phi)_n(f)"]
	\ar[r,shift right=.75ex,swap,"\mathbf{UGR}(\Phi)_n(g)"]&
	(k,1)
\end{tikzcd}
\]
is a $(P',L')$-admissible pair in $\mathbf{UGR}(\Omega')_n$. Lemma~\ref{uni_prop_of_ext} therefore yields a unique morphism
\[
\mathbf{UGR}(\Phi)_{n+1}:
\mathbf{UGR}(\Omega)_{n+1}
\longrightarrow
\mathbf{UGR}(\Omega')_{n+1}
\]
extending $\mathbf{UGR}(\Phi)_n$ and satisfying
\[
\mathbf{UGR}(\Phi)_{n+1}(\delta_{f,g})=\delta_{\mathbf{UGR}(\Phi)_n(f),\,\mathbf{UGR}(\Phi)_n(g)}.
\]

Passing to colimits, the universal property induces a unique morphism
\[
\mathbf{UGR}(\Phi):
\mathbf{UGR}(\Omega)
\longrightarrow
\mathbf{UGR}(\Omega')
\]
of $\infty$-Lawvere theories. Consequently, we obtain a functor
\[
\mathbf{UGR}:\mathbf{cTh}\longrightarrow\infty\mathbf{Law}.
\]

Finally, by taking colimits, we extend the domain of $\mathbf{UGR}$ from controlled theories to connected diagrams of controlled theories.

\subsection*{The Reduced Grothendieck Realization}
Given a controlled theory $\Omega$, we now describe a reduced version of the Grothendieck realization. Consider the composite
\[
\Theta_0^\mathbf{op}\to AC\to \mathcal{F}^\infty\to \mathbf{UGR}(\Omega)_0,
\]
which we denote by $\overline{(-,1)}_0$. Observe that, in constructing $\mathbf{UGR}(\Omega)$, duplicate lifts are added for the $\infty$-groupoid portion of the structure. The reduced Grothendieck realization is obtained by identifying these duplicate lifts with the original ones.

Inductively, the map $\overline{(-,1)}_0$ extends uniquely to a morphism
\[
AC\xrightarrow{\overline{(-,1)}}\mathbf{UGR}(\Omega)
\]
that sends each lift added in the construction of $AC$ to the corresponding lift added in the construction of $\mathbf{UGR}(\Omega)$. Consequently, we obtain parallel morphisms
\[
\begin{tikzcd}
	AC\ar[r,shift left=.75ex,"\overline{(-{,}1)}"]
	\ar[r,shift right=.75ex,swap,"(-{,}1)"]&
	\mathbf{UGR}(\Omega).
\end{tikzcd}
\]

These induce the parallel morphisms
\[
\begin{tikzcd}
	\mathcal{F}^{\infty}\ar[r,shift left=.75ex]
	\ar[r,shift right=.75ex,swap]&
	\mathbf{UGR}(\Omega).
\end{tikzcd}
\]

We define the \emph{reduced Grothendieck realization} of $\Omega$, denoted $\mathbf{GR}(\Omega)$,
to be the coequalizer of this pair. This construction is functorial, yielding a functor
\[
\mathbf{GR}:\mathbf{cTh}\to\infty\mathbf{Law}.
\]
As before, we extend the domain of $\mathbf{GR}$ to connected diagrams of controlled theories by taking colimits.

At present, we are not aware of any structural advantage to considering the reduced Grothendieck realization. It is introduced primarily because it more closely resembles the behavior of familiar constructions in low-dimensional category theory, where operations such as the composition of $1$-cells are required to have a canonical choice. In subsequent work, we will primarily use the unreduced Grothendieck realization.

\subsection*{Higher Categorical Algebra and Group Completion}
We now provide globular models for the following structures:
\begin{itemize}
    \item monoidal $\infty$-groupoids,
    \item symmetric monoidal $\infty$-groupoids,
    \item coherent $\infty$-groups, and
    \item Picard $\infty$-groupoids.
\end{itemize}

\begin{definition}\label{monoid_inf_grpds}
A model
\[
A:\mathbf{UGR}(\Omega_{mon})\to\mathbf{Set}
\]
over $\mathbf{UGR}(\Omega_{mon})$ is called a \emph{monoidal $\infty$-groupoid}, where $\Omega_{mon}$ is the controlled theory defined in Example~4.3 of \cite{Taylor2026Controlled}.
\end{definition}

\begin{notation}
We write $\mathbf{MonGpd}^\infty$ for the category of models
\[
\mathbf{Mod}(\mathbf{UGR}(\Omega_{mon})).
\]
\end{notation}

\begin{definition}\label{sym_monoid_inf_grpds}
A model
\[
A:\mathbf{UGR}(\Omega_{cm})\to\mathbf{Set}
\]
over $\mathbf{UGR}(\Omega_{cm})$ is called a \emph{symmetric monoidal $\infty$-groupoid}, where $\Omega_{cm}$ is the controlled theory defined in Example~4.4 of \cite{Taylor2026Controlled}.
\end{definition}

\begin{notation}
We write $\mathbf{SMG}^\infty$ for the category of models
\[
\mathbf{Mod}(\mathbf{UGR}(\Omega_{cm})).
\]
\end{notation}

\begin{definition}\label{inf_groups}
A model
\[
A:\mathbf{UGR}(\Omega_{grp})\to\mathbf{Set}
\]
over $\mathbf{UGR}(\Omega_{grp})$ is called a \emph{coherent $\infty$-group}, where $\Omega_{grp}$ is the controlled theory defined in Example~4.5 of \cite{Taylor2026Controlled}.
\end{definition}

\begin{notation}
We write $\mathbf{Grp}^\infty$ for the category of models
\[
\mathbf{Mod}(\mathbf{UGR}(\Omega_{grp})).
\]
\end{notation}

We now turn our attention to Picard $\infty$-groupoids.

\begin{definition}\label{def_of_pic_inf_grpd}
A model
\[
A:\mathbf{UGR}(\Gamma_{pic})\to\mathbf{Set}
\]
over $\mathbf{UGR}(\Gamma_{pic})$ is called a \emph{Picard $\infty$-groupoid}, where $\Gamma_{pic}$ is the connected diagram of controlled theories defined in Example~4.9 of \cite{Taylor2026Controlled}.
\end{definition}

\begin{notation}
We write $\mathbf{Pic}^\infty$ for the category of models
\[
\mathbf{Mod}(\mathbf{UGR}(\Gamma_{pic})).
\]
\end{notation}

By cocompleteness of categories of models, we obtain the following free-forgetful adjunctions, where the left adjoints are given by Kan extension.
\[
\begin{tikzpicture}[node distance=4cm]
	\node (A) {$\mathbf{Fr}:\mathbf{MonGpd}^{\infty}$};
	\node (B)[right of=A]{$\infty\mathbf{Grp}:U$};
	\draw[transform canvas={yshift=0.3ex},->] (A) to node {} (B);
	\draw[transform canvas={yshift=-0.3ex},->,swap] (B) to node {} (A);
\end{tikzpicture}
\]

\[
\begin{tikzpicture}[node distance=4cm]
	\node (A) {$\mathbf{Fr}:\mathbf{SMG}^{\infty}$};
	\node (B)[right of=A]{$\mathbf{Pic}^{\infty}:U$};
	\draw[transform canvas={yshift=0.3ex},->] (A) to node {} (B);
	\draw[transform canvas={yshift=-0.3ex},->,swap] (B) to node {} (A);
\end{tikzpicture}
\]

By monadicity, we obtain the following.

\begin{lemma}
The following adjunctions are monadic.
\[
\begin{tikzpicture}[node distance=4cm]
	\node (A) {$\mathbf{Fr}:\mathbf{MonGpd}^{\infty}$};
	\node (B)[right of=A]{$\infty\mathbf{Grp}:U$};
	\draw[transform canvas={yshift=0.3ex},->] (A) to node {} (B);
	\draw[transform canvas={yshift=-0.3ex},->,swap] (B) to node {} (A);
\end{tikzpicture}
\]

\[
\begin{tikzpicture}[node distance=4cm]
	\node (A) {$\mathbf{Fr}:\mathbf{SMG}^{\infty}$};
	\node (B)[right of=A]{$\mathbf{Pic}^{\infty}:U$};
	\draw[transform canvas={yshift=0.3ex},->] (A) to node {} (B);
	\draw[transform canvas={yshift=-0.3ex},->,swap] (B) to node {} (A);
\end{tikzpicture}
\]
\end{lemma}

\begin{definition}
Let
\[
(T,\eta,\mu)
\]
and
\[
(T^\Sigma,\eta^\Sigma,\mu^\Sigma)
\]
be the monads induced on
\[
\mathbf{MonGpd}^\infty
\]
and
\[
\mathbf{SMG}^\infty,
\]
respectively. We call both of these monads the \emph{group completion monad}.
\end{definition}

\section{Homotopy theory of models}
Before we begin this section, we note that we obtained this section by modifying results of Section 4 of \cite{Dim1} and Subsection 5.3 of \cite{Henry2016}.
\subsection*{Path Components}
Given an $\infty$-groupoid $X$, we define $\pi_0(X)$ to be the quotient of $X_0$ by the smallest equivalence relation $\sim$ such that $x\sim x'$ whenever there exists $f\in X_1$ such that
\[
s(f)=x,\qquad t(f)=x'.
\]
This construction defines a functor $\pi_0:\infty\mathbf{Gpd}\to\mathbf{Set}$. Furthermore, using the equivalence
\[
\mathbf{Mod}(\mathcal{F}^\infty)\simeq\infty\mathbf{Gpd},
\]
we obtain a functor 
\[
\mathbf{Mod}(\mathcal{F}^\infty)\xrightarrow{\simeq}\infty\mathbf{Gpd}\xrightarrow{\pi_0}\mathbf{Set}
\]
that we also denote by $\pi_0$.

Let $H:\mathcal{F}^\infty\to V$ be an $\infty$-Lawvere theory. Define the Lawvere theory $\overline{V}_0$ by setting the homset $\overline{V}_0(n,k)$ to be the quotient of
\[
V((0,n),(0,k))
\]
by the equivalence relation generated by declaring $f\sim g$ if there exists $\alpha\in V((0,n),(1,k))$ such that
\[
(s,1)\circ\alpha=f,\qquad (t,1)\circ\alpha=g.
\]

\begin{lemma}
Let $H:\mathcal{F}^\infty\to V$ be an $\infty$-Lawvere theory. The path object functor
\[
\pi_0:\mathbf{Mod}(\mathcal{F}^\infty)\to\mathbf{Set}
\]
lifts to a functor
\[
\pi_0:\mathbf{Mod}(V)\to\mathbf{Mod}(\overline{V}_0).
\]
\end{lemma}

\subsection*{Homotopy Group Functors}

Let $X$ be an $\infty$-groupoid, let $x\in X_0$, and let $n\geq 1$. Define $\pi_n(X,x)$ to be the group obtained by quotienting the set
\[
\{f\in X_n:s(f)=t(f)=Z^{n-1}(x)\}
\]
by the smallest equivalence relation $\sim$ such that $f\sim f'$ whenever there exists $\alpha\in X_{n+1}$ such that
\[
s(\alpha)=f,\qquad t(\alpha)=f'.
\]

This construction induces a functor
\[
\mathbf{Mod}(\mathcal{F}^\infty)\xrightarrow{\simeq}\infty\mathbf{Gpd}\xrightarrow{\pi_1}\mathbf{Grp}
\]
and a functor
\[
\mathbf{Mod}(\mathcal{F}^\infty)\xrightarrow{\simeq}\infty\mathbf{Gpd}\xrightarrow{\pi_n}\mathbf{Ab}
\]
for $n\geq 2$. Let $H:\mathcal{F}^\infty\to V$ be an $\infty$-Lawvere theory. We obtain induced functors
\[
\mathbf{Mod}(V)\xrightarrow{H^*}\mathbf{Mod}(\mathcal{F}^\infty)\xrightarrow{\pi_1}\mathbf{Grp}
\]
and
\[
\mathbf{Mod}(V)\xrightarrow{H^*}\mathbf{Mod}(\mathcal{F}^\infty)\xrightarrow{\pi_n}\mathbf{Ab}
\]
for $n\geq 2$.

\subsection*{Weak Equivalences and Generating Data}
For this subsection, we fix an $\infty$-Lawvere theory $H:\mathcal{F}^\infty\to V$.

\begin{definition}
A morphism $f:X\to Y$ of $V$-models is called a \emph{weak equivalence} if the following conditions are satisfied.
\begin{itemize}
	\item The map $\pi_0(f):\pi_0(X)\to\pi_0(Y)$ is an isomorphism.
	\item For all $n\geq 1$ and every $x\in X_0$, the morphism
	\[
	\pi_n(X,x)\to\pi_n(Y,f(x)),
	\]
	induced by $f$, is an isomorphism.
\end{itemize}
\end{definition}

\begin{notation}
We write $\mathbf{W}^V$ for the class of weak equivalences of $V$-models.
\end{notation}

We now define $D^k_V$ to be the free $V$-model generated by a single $k$-cell and $S^{k-1}_V$ to be the free $V$-model generated by a parallel pair of $(k-1)$-cells for all $k\geq 0$.

\begin{definition}
The set of \emph{generating cofibrations} (respectively, \emph{generating trivial cofibrations}) consists of the boundary inclusions
$\mathbf{I}^V\overset{def}{=}\{j^V_k:S_V^{k-1}\rightarrow D_V^k\}_{k\geq 0}$
(respectively, the source maps
$\mathbf{J}^V\overset{def}{=}\{\sigma_k:=D(s):D_k\rightarrow D_{k+1}\}_{k\geq 0}$).
\end{definition}

\begin{definition}
A map $f:X\to Y$ is called a \emph{(trivial) fibration} if it has the right lifting property with respect to the set $\mathbf{J}$ (respectively, $\mathbf{I}$).
\end{definition}
\subsection*{Hypothetical Semi-Model Structures}
Let $H:\mathcal{F}^\infty\to V$ be an $\infty$-Lawvere theory.

\begin{definition}
The \emph{canonical semi-model structure on $\mathbf{Mod}(V)$}, if it exists, is the semi-model structure with
\begin{itemize}
    \item weak equivalences given by the class $\mathbf{W}^V$,
    \item generating cofibrations given by $\mathbf{I}^V$, and
    \item generating trivial cofibrations given by $\mathbf{J}^V$.
\end{itemize}
\end{definition}

In the case $(V,H)=(\mathcal{F}^\infty,\mathbf{id}_{\mathcal{F}^\infty})$, we recover $\infty$-groupoids and the canonical semi-model structure introduced by Henry in \cite{Henry2016}. If they exist, all other canonical semi-model structures are obtained by transferring the canonical semi-model structure from $\infty$-groupoids.

We now introduce the following conjecture.

\begin{conjecture}[The Generalized Pushout Conjecture]\label{Gen_PO_Conj}
The following statements hold.
\begin{itemize}
    \item Given a cofibrant $\infty$-groupoid $X$ and a pushout
    \[
    \begin{tikzpicture}[node distance=2.5cm]
        \node(A) {$D^n$};
        \node (B) [right of=A]{$X$};
        \node(C)[below of=A]{$D^{n+1}$};
        \node (D)[right of=C]{$X_+$};
        \draw[->](B) to node[black,right=3]{$p$}(D);
        \draw[->](A) to node {}(B);
        \draw[->] (A) to node {}(C);
        \draw[->](C) to node {}(D);
    \end{tikzpicture},
    \]
    the map $p:X\to X_+$ is a weak equivalence.
    \item Given a connected diagram of controlled theories $J$, a model $X$ in $\mathbf{Mod}(\mathbf{UGR}(\Gamma))$ with $UX$ cofibrant, and a pushout
    \[
    \begin{tikzpicture}[node distance=2.5cm]
        \node(A) {$\mathbf{Fr} (D^n)$};
        \node (B) [right of=A]{$X$};
        \node(C)[below of=A]{$\mathbf{Fr}(D^{n+1})$};
        \node (D)[right of=C]{$X_+$};
        \draw[->](B) to node[black,right=3]{$p$}(D);
        \draw[->](A) to node {}(B);
        \draw[->] (A) to node {}(C);
        \draw[->](C) to node {}(D);
    \end{tikzpicture}
    \]
    of diagrams, the map $p:X\to X_+$ is a weak equivalence.
\end{itemize}
\end{conjecture}

We now show that this conjecture is the only remaining obstruction to establishing a full connection between algebraic and homotopical models of higher category theory.

\begin{theorem}
Assuming the Generalized Pushout Conjecture, the semi-model structure exists on $\infty\mathbf{Gpd}$ and the Homotopy Hypothesis holds. Moreover, the semi-model structure exists on $\mathbf{Mod}(\mathbf{UGR}(\Gamma))$.
\end{theorem}

\begin{proof}
The first statement follows directly from Corollary 5.3.13 of \cite{Henry2016}. The second statement follows by combining Theorem 2.2.2 of \cite{BataninWhite2022} with Theorem \ref{forgetful functor is monadic}.
\end{proof}

\begin{references*}

\bibitem[Ad\'{a}mek and Rosick\'{y}, 1994]{AdRo}
J.~Ad\'{a}mek and J.~Rosick\'{y},
\emph{Locally Presentable and Accessible Categories}.
Cambridge University Press, Cambridge, 1994.

\bibitem[Ara, 2013]{Dim1}
D.~Ara,
\emph{On the homotopy theory of Grothendieck $\infty$-groupoids}.
J. Pure Appl. Algebra \textbf{217} (2013), no.~7, 1237--1278.

\bibitem[Batanin and White, 2022]{BataninWhite2022}
M.~Batanin and D.~White,
\emph{Homotopy Theory of Algebras of Substitudes and Their Localisation}.
Trans. Amer. Math. Soc. \textbf{375} (2022), no.~5, 3569--3640.

\bibitem[Bourke, 2016]{Bo2}
J.~Bourke,
\emph{Note on the construction of globular weak omega-groupoids from
types, topological spaces, $\ldots$}.
Cahiers Topologie G\'eom. Diff\'erentielle Cat\'eg. \textbf{57}
(2016), no.~4, 281--294.

\bibitem[Bressie, 2020]{Bressie2020}
P.~M.~Bressie,
\emph{The $\omega$-categorification of Algebraic Theories}.
Available at arXiv:2006.07191 [math.CT], 2020.

\bibitem[Cheng, 2011]{Cheng2011}
E.~Cheng,
\emph{Iterated Distributive Laws}.
Math. Proc. Cambridge Philos. Soc.
\textbf{150} (2011), no.~3, 459--487.

\bibitem[Garner, 2009b]{Garner2009}
R.~Garner,
\emph{Understanding the Small Object Argument}.
Appl. Categ. Structures
\textbf{17} (2009), no.~3, 247--285.

\bibitem[Grothendieck, 2022]{Gr1}
A.~Grothendieck,
\emph{Pursuing Stacks (\`a la poursuite des champs). Vol.~I}.
Documents Math\'ematiques~20.
Soci\'et\'e Math\'ematique de France, Paris, 2022.

\bibitem[Henry, 2016]{Henry2016}
S.~Henry,
\emph{Algebraic Models of Homotopy Types and the Homotopy Hypothesis}.
Available at arXiv:1609.04622 [math.AT], 2016.

\bibitem[Henry and Lanari, 2023]{HenLan}
S.~Henry and E.~Lanari,
\emph{On the homotopy hypothesis for 3-groupoids}.
Theory Appl. Categ.
\textbf{39} (2023), Paper No.~26, 735--768.

\bibitem[Lanari, 2018]{LanE}
E.~Lanari,
\emph{A semi-model structure for Grothendieck weak 3-groupoids}.
Preprint, 2018.

\bibitem[Lawvere, 1963]{Law2}
F.~W.~Lawvere,
\emph{Functorial semantics of algebraic theories}.
Proc. Nat. Acad. Sci. U.S.A.
\textbf{50} (1963), 869--872.

\bibitem[Leinster, 2004]{Lein}
T.~Leinster,
\emph{Higher Operads, Higher Categories}.
London Mathematical Society Lecture Note Series~298.
Cambridge University Press, Cambridge, 2004.

\bibitem[Maltsiniotis, 2005]{Geo2}
G.~Maltsiniotis,
\emph{La th\'eorie de l'homotopie de Grothendieck}.
Ast\'erisque \textbf{301} (2005), vi+140 pp.

\bibitem[Maltsiniotis, 2010]{Malt}
G.~Maltsiniotis,
\emph{Grothendieck $\infty$-groupoids, and still another definition of
$\infty$-categories}.
Available at arXiv:1009.2331 [math.CT], 2010.

\bibitem[Taylor, 2026a]{Taylor2026}
J.~Taylor,
\emph{Controlled Theories and Infinity Lawvere Theories:
Weakening Axioms in the Globular Setting}.
Ph.D. thesis,
Case Western Reserve University,
2026.

\bibitem[Taylor, 2026b]{Taylor2026Controlled}
J.~Taylor,
\emph{Controlled theories, categorification, and homotopification}.
Available at arXiv:2607.24716 [math.CT], 2026.
\end{references*}

\end{document}